\documentclass[twoside,a4paper,reqno,11pt]{amsart}
\usepackage{amsfonts, amsbsy, amsmath, amsthm, amssymb, latexsym, verbatim, enumerate}
\usepackage{mathrsfs}
\usepackage[top=30mm,right=30mm,bottom=30mm,left=30mm]{geometry}

\usepackage{bm}
\usepackage[pdftex]{color,graphicx}

\makeatletter
\newcommand{\imod}[1]{\allowbreak\mkern4mu({\operator@font mod}\,\,#1)}
\makeatother

\def\hal{\unskip\nobreak\hfill\penalty50\hskip10pt\hbox{}\nobreak

\hfill\vrule height 5pt width 6pt depth 1pt\par\vskip 2mm}

\numberwithin{equation}{section}

\def\a{\alpha}
 
 \def\e{\epsilon}
 
  \def\D{\Delta}

 \def\c{\chi}

 \def\N{\mathbb N}
 
 \def\Z{\mathbb Z}
 \def\C{\mathbb C}
 \def\F{\mathbb F}
 \def\G{\Gamma}

 \def\la{\langle}

 \def\diam{\mathsf {diam}}
 
 \def\L{{\mathcal L}}

 \def\PSL{{\rm PSL}}
 \def\SL{{\rm SL}}

 \def\l{\lambda}
\def\D{\Delta}

 \def\bmax{\mathsf{b}}
 \def\dmin{\mathsf{d}}
 \def\St{\mathsf{St}}
 \def\SSS{\mathsf{S}}
 \def\AAA{\mathsf{A}}
 \def\Irr{\mathrm{Irr}}
 \def\Ind{\mathrm{Ind}}
 \def\Z{\mathbb{Z}}
 \def\GSS{G_{\mathrm {ss}}}
 \def\CB{\mathbf{C}}
 \def\ZB{\mathbf{Z}}
 \def\GC{\mathcal{G}}
 \def\SL{\mathrm {SL}}
 \def\GL{\mathrm {GL}}
 \def\PSL{\mathrm {PSL}}
 \def\GU{\mathrm {GU}}
 \def\SU{\mathrm {SU}}
  \def\PSU{\mathrm {PSU}}
 
 \def\MC{\mathcal {M}}
 \def\pf{\noindent {\bf Proof.$\;$ }}

\def\hal{\unskip\nobreak\hfil\penalty50\hskip10pt\hbox{}\nobreak
 \hfill\vrule height 5pt width 6pt depth 1pt\par\vskip 2mm}

\newtheorem{theorem}{Theorem}
\newtheorem*{conj*}{Conjecture}

\newtheorem{conj}[theorem]{Conjecture}
\newtheorem{corol}[theorem]{Corollary}

\newtheorem{thm}{Theorem}[section]
\newtheorem{prop}[thm]{Proposition}
\newtheorem{lem}[thm]{Lemma}

\theoremstyle{definition}

\begin{document}

 \author{Martin W. Liebeck}
\address{M.W. Liebeck, Department of Mathematics,
    Imperial College, London SW7 2BZ, UK}
\email{m.liebeck@imperial.ac.uk}

\author{Aner Shalev}
\address{A. Shalev, Institute of Mathematics, Hebrew University, Jerusalem 91904, Israel}
\email{shalev@math.huji.ac.il}

\author[P. H. Tiep]{Pham Huu Tiep}
\address{P.H. Tiep, Department of Mathematics, Rutgers University, Piscataway, NJ 08854, USA}
\email{tiep@math.rutgers.edu}

\title{On the diameters of McKay graphs for finite simple groups}

\maketitle

\begin{abstract}
Let $G$ be a finite group, and $\a$ a nontrivial  character of $G$. The McKay graph $\MC(G,\a)$ has the irreducible characters of $G$ as vertices, with an edge from $\c_1$ to $\c_2$ if $\c_2$ is a constituent of $\a\c_1$. We study the diameters of McKay graphs for simple groups $G$. For $G$ a group of Lie type, we show that for any $\a$, the diameter is bounded by a quadratic function of the rank, and obtain much stronger bounds for $G=\PSL_n(q)$ or $\PSU_n(q)$. We also bound the diameter for symmetric and alternating groups.
 \end{abstract}



\footnote{The second author acknowledges the support of ISF grant 686/17 and the Vinik chair of mathematics which he holds. The third author gratefully acknowledges the support of the
NSF (grant DMS-1840702) and the Joshua Barlaz Chair in Mathematics. The second and the third authors were also partially supported by BSF grant 2016072.
The authors also acknowledge the support of the National Science Foundation under Grant No. DMS-1440140 while they were in residence at the Mathematical Sciences Research Institute in Berkeley, California, during the Spring 2018 semester. }

\section{Introduction}

For a finite group $G$, and a (complex) character $\a$ of $G$, the {\it McKay graph} $\MC(G,\a)$ is defined to be the directed graph with vertex set ${\rm Irr}(G)$, there being an edge from $\c_1$ to $\c_2$ if and only if $\c_2$ is a constituent of $\a\c_1$.
The famous McKay correspondence \cite{M} shows that if $G$ is a finite subgroup of $\SU_2(\C)$ and $\a$ is the corresponding 2-dimensional character of $G$, then $\MC(G,\a)$ is an affine Dynkin diagram of type $A$, $D$ or $E$. The purpose of this  paper is to initiate the study of McKay graphs for simple groups, focussing particularly on their diameters.

By a classical result of Burnside and Brauer \cite{Br}, $\MC(G,\a)$ is connected if and only if $\a$ is faithful, and moreover in this case an upper bound for the diameter $\diam \MC(G,\a)$  is given by $N-1$, where $N$ is the number of distinct values of $\a$. (Indeed, in this case $\sum^{N-1}_{j=0}\a^j$
contains every irreducible character of $G$. Taking $\beta$ to be an irreducible constituent of $\bar\chi_1\chi_2$, we can find $0 \leq j \leq N-1$ such
that
$$0 < [\a^j,\beta]_G \leq [\a^j,\bar\chi_1\chi_2]_G = [\a^j\chi_1,\chi_2]_G,$$
i.e. a directed path of length $j$ connects $\chi_1$ to $\chi_2$.)

An obvious lower bound for $\diam \MC(G,\a)$ (when $\a(1) > 1$) is given by $\frac{\log \bmax(G)}{\log \a(1)}$, where $\bmax(G)$ is the largest degree of an irreducible character of $G$. One can do slightly better, by observing that if $d := \diam(M,\a)$, then
$$2\a(1)^d > \sum^d_{i=0}\a(1)^i \geq \sum_{\c \in \Irr(G)}\c(1) > \biggl( \sum_{\c \in \Irr(G)}\c(1)^2 \biggr)^{1/2} = |G|^{1/2}.$$
It follows that
\[
\diam \MC(G,\a) \ge \tfrac{1}{2} \frac{\log (|G|/4)}{\log \a(1)}.
\]
This bound is far from tight for many groups $G$. However, for finite simple groups we make the following conjecture.

\begin{conj}\label{diamconj} There is an absolute constant $C$ such that for any finite non-abelian simple group $G$, and any nontrivial irreducible character $\a$ of $G$,
\[
\diam \MC(G,\a) \le C\frac{\log |G|}{\log \a(1)}.
\]
\end{conj}

Note that \cite{LS01} gives the analogous bound for conjugacy classes: namely, for a nontrivial conjugacy class $S$ of a finite (non-abelian) simple group $G$, we have $\diam  \G(G,S) \le C\frac{\log |G|}{\log |S|}$, where $\G(G,S)$ is the Cayley graph of $G$ with respect to $S$.

In general, Conjecture \ref{diamconj} cannot hold for arbitrary faithful character of $G$. However, once it holds for
faithful irreducible characters, then it also holds for all faithful {\it multiplicity-free} characters, albeit with a different constant $C$.
To see this, note that Theorem 1.1 of \cite{LS05} implies that the number $r_m(G)$ of irreducible characters of degree $m$ of a non-abelian
finite simple group $G$ satisfies $r_m(G) = o(m^{1+\e})$ for any fixed $\e > 0$. This implies that $G$ has at most $m^c$ irreducible characters
of degree at most $m$, where $c$ is an absolute constant. Now let $\beta$ be a faithful multiplicity-free character of $G$, and let
$\a$ be an irreducible constituent of $\beta$ of maximal degree. Then $\beta(1) \le \a(1)^{c+1}$, so assuming Conjecture \ref{diamconj}
we obtain
\[
\diam \MC(G,\beta) \le \diam \MC(G,\a) \le C\frac{\log |G|}{\log \a(1)} \le C(c+1) \frac{\log |G|}{\log \beta(1)}.
\]

In this paper we prove Conjecture \ref{diamconj} for many families of simple groups of Lie type.

\begin{theorem}\label{bdd} There is an absolute constant $C$ such that $\diam \MC(G,\a) \le Cr^2$
for any finite simple group $G$ of Lie type of rank $r$ and any nontrivial irreducible character $\a$ of $G$.
Hence Conjecture \ref{diamconj} holds for simple groups of Lie type of bounded rank.
\end{theorem}

Our proof of Theorem \ref{bdd} shows that in fact one can take $C = 489$.

For classical groups of unbounded rank, we are able to handle the projective special linear and unitary groups $\PSL_n^\e (q)$ (with
$\PSL^+ = \PSL$ and $\PSL^- = \PSU$), where $q$ is large compared to $n$. The proof uses major new advances in the theory of character ratios, taken from \cite{BLST, TT}.

\begin{theorem}\label{diam}  There exist an absolute constant $C$ and  a function $g:\N\to \N$ such that the following holds.
If $G=\PSL_n^\e (q)$ with $n \geq 2$, $\e = \pm$, and $q>g(n)$, then
\[
\diam \MC(G,\a) < C\frac{\log |G|}{\log \a(1)},
\]
for all nontrivial irreducible characters  $\a$ of $G$.
\end{theorem}

As one can see from our proof,  $C$ can be taken to be $15$, and $g$ can also be made explicit.

Theorem \ref{diam} gives rise to the following extension.

\begin{corol}\label{diam-mf}
With the function $g(n)$ and $C$ as in Theorem \ref{diam}, we have
\[
\diam \MC(G,\a) < 2.5C\frac{\log |G|}{\log \a(1)},
\]
for any simple group $G=\PSL_n^\e (q)$ with $n \geq 2$, $\e = \pm$, $q>\max\{g(n),11\}$, and for any faithful multiplicity-free character $\a$ of $G$.
\end{corol}

For other types of classical groups of unbounded ranks, we have not yet been able to prove the conjecture, but we do have results bounding the diameter by a linear function of the rank; these will appear in a forthcoming paper.

For alternating (and symmetric) groups, the conjecture is wide open. However we are able to establish the following bound.

\begin{theorem}\label{alt} Let $n\ge 5$ and let $G = \AAA_n$ or $\SSS_n$. Then
for any faithful irreducible character $\a$ of $G$, we have $\diam \MC(G,\a) \le 4n-4$.
\end{theorem}

This bound is essentially best possible (up to the multiplicative constant); it confirms Conjecture \ref{diamconj} for faithful 
irreducible characters $\a$ of $G = \AAA_n, \SSS_n$
of polynomial degree, namely $\a(1) = O(n^{O(1)})$. It also follows immediately that Theorems \ref{bdd} and \ref{alt}
extend to all faithful multiplicity-free characters $\a$ of $G$.

This paper is organized as follows. In Section \ref{bddsec} we prove Theorem \ref{bdd}.
Section \ref{pslsec} is devoted to the proof of Theorem \ref{diam} and Corollary \ref{diam-mf},
using new developments in character bounds (see \cite{BLST, TT}). Theorem \ref{alt} is proved
in Section \ref{altsec}, and in Section \ref{qssec} we briefly discuss the diameter of
McKay graphs for quasi-simple groups.

\section{Preliminaries and groups of bounded rank}\label{bddsec}

Our proof uses the following results, taken from \cite{HSTZ} and \cite{Gl}.

\begin{prop}\label{stsq} Let $G$ be a finite simple group of Lie type, and let $\St$ denote the Steinberg character of
$G$. Then provided $G$ is not a unitary group in odd dimension,  $\St^2$ contains every irreducible character of $G$ as a constituent.
In all cases, $\St^3$ contains every irreducible character.
\end{prop}

\pf
The first statement is \cite[Theorem 1.2]{HSTZ}. Consider the exceptional case $G = \PSU_n(q)$
with $2 \nmid n \geq 3$. Then, again by \cite[Theorem 1.2]{HSTZ}, $\St^2$ contains all $\chi \in \Irr(G)$ but the unique unipotent character
$\alpha$ of degree $(q^n-q)/(q+1)$. Let $\chi \in \Irr(G)$ and suppose that $\St\cdot\bar\chi$ is a multiple of $\alpha$:
$$\St\cdot\bar\chi = k\alpha$$
for some $k \in \Z$. Then $k = \St(1)\chi(1)/\alpha(1) \neq 0$, and so $\alpha(t) = \St(t)\cdot\bar\chi(t)/k = 0$ for any transvection $t \in G$. However,
$\alpha(t) = -(q^n-q(-1)^n)/(q+1) \neq 0$ by \cite[Lemma 4.1]{TZ2}, a contradiction. Hence $\St\cdot\bar\chi$ contains some character
$\beta \in \Irr(G) \smallsetminus \{\alpha\}$. It follows that
$$0 < [\St\cdot\bar\chi,\beta]_G \leq [\St\cdot\bar\chi,\St^2]_G = [\St^3,\chi]_G,$$
i.e. $\chi$ is an irreducible constituent of $\St^3$.

A similar argument as above shows that $\St^3$ contains all irreducible characters of $G$, for any simple group $G$ of Lie type.
\hal

\begin{prop}\label{gluck} {\rm \cite{Gl}} Let $G$ be a finite simple group of Lie type over $\F_q$, and let $1 \ne g \in G$.
Then for any $\c \in {\rm Irr}(G)$,
\[
\frac{|\c(g)|}{\c(1)} \le {\rm min}\left(\frac{3}{\sqrt{q}},\frac{19}{20}\right).
\]
\end{prop}

We can now prove Theorem \ref{bdd}. Let $G$ be a simple group of Lie type over a field $\F_q$ (of characteristic $p$) of rank $r$, and let $\GSS$ denote the set of semisimple elements of $G$. Recall (see \cite[6.4.7]{Car}) that the values of the Steinberg character $\St$ are
\begin{equation}\label{stval}
\St(g) = \left\{\begin{array}{l} \e_g |\CB_G(g)|_p, \hbox{ if }g \in \GSS, \\
0, \hbox{ if } g \not \in \GSS,
\end{array}
\right.
\end{equation}
where $\e_g=\pm 1$.

\begin{lem}\label{stdiam}
There is an absolute constant $D$ such that for any $l\ge Dr^2$ and any $\c \in {\rm Irr}(G)$, we have
$[\c^l,\St]_G\ne 0$. Indeed, $D=163$ suffices.
\end{lem}

\begin{proof}
By (\ref{stval}),
\begin{equation}\label{useag}
\begin{array}{ll}
[\c^l,\St]_G & = \dfrac{1}{|G|}\sum_{g\in \GSS} \e_g \c^l(g)|\CB_G(g)|_p \\
               & = \dfrac{\chi^l(1)}{|G|}\left(|G|_p + \sum_{1 \neq g \in \GSS} \e_g \left(\frac{\chi(g)}{\chi(1)}\right)^l|\CB_G(g)|_p\right).
\end{array}
\end{equation}
 Hence $[\c^l,\St]_G \ne 0$ provided $\Sigma_l < |G|_p$, where
\[
\Sigma_l := \sum_{1 \neq g\in \GSS}  \left|\frac{\chi(g)}{\chi(1)}\right|^l|\CB_G(g)|_p.
\]
Note that $|G| < q^{4r^2}$. Assume first that $q>9$. Then Proposition \ref{gluck} implies that $\Sigma_l < |G|_p$ provided
$q^{4r^2}\cdot (3/q^{1/2})^l < 1$, which holds if $l \ge 96r^2$. For $q\le 9$ we need
$q^{4r^2}\cdot (19/20)^l < 1$, and this holds when $l\ge 163r^2$.
\end{proof}

Now let $1\ne \a \in {\rm Irr}(G)$. It follows from Lemma \ref{stdiam} and Proposition \ref{stsq} that
$\a^{3Dr^2}$ contains all irreducible characters of $G$. Hence, given any two $\chi_1,\chi_2 \in \Irr(G)$,
$$0 \neq \left[\a^{3Dr^2},\bar\chi_1\chi_2\right]_G = \left[\a^{3Dr^2}\chi_1,\chi_2\right]_G,$$
i.e. a directed path of length $\leq 3Dr^2$ connects $\chi_1$ to $\chi_2$ in $\MC(G,\a)$.
We conclude that $\diam \MC(G,\a) \le 3Dr^2$,  completing the proof of Theorem \ref{bdd}.

\section{Projective special linear and unitary groups}\label{pslsec}
Throughout this section, which is devoted to prove Theorem \ref{diam},
let $G=\PSL_n^\e (q)$ with $\e = \pm$.  For a semisimple element $g \in \GSS$, let $\hat g$ be a preimage of $g$ in $\SL_n^\e (q)$, and define $\nu(g) = {\rm supp}(\hat g)$, the codimension of the largest eigenspace of $\hat g$ over $\bar \F_q$.

We shall need the following bound for character ratios of semisimple elements, which follows from the deep results in \cite{BLST, TT}.

\begin{thm} \label{char} There is a function $f:\N\to \N$ such that for any $g \in \GSS$ with $s = \nu(g)$, and any $\c \in {\rm Irr}(G)$, we have
\[
|\c(g)| < f(n)\c(1)^{1-\frac{s}{n}}.
\]
\end{thm}

\begin{proof}
Let $\GC = \SL_n(K)$, $K=\bar \F_q$ be the ambient algebraic group with $G = \GC^F/\ZB(\GC^F)$, where $F$ is a Frobenius endomorphism. Then $\CB_{\GC}(\hat g) = \L:= \tilde\L \cap \GC$, where $\tilde\L = \prod_{i=1}^m \GL_{n_i}(K)$, $1 \leq n_1\le \cdots \le n_m$, and $\sum^m_{i=1}n_i =n$.
Note that $\nu(g)=s = n-n_m$;
and that $\L$ is $F$-stable (but not necessarily split).

We now apply \cite[Cor. 1.11(c)]{TT} (which is an extension of \cite[Thm. 1.1]{BLST}). That gives a function $f:\N\to \N$ such that for any $\c \in {\rm Irr}(G)$,
\[
|\c(g)| < f(n)\c(1)^{\a(\L)},
\]
where $\a(\L)$ is the maximum value of $\frac{\dim u^{\L}}{\dim u^{\GC}}$ over nontrivial unipotent elements $u\in \L$ (and $\a(\L) = 0$ if $\L$ is a torus).
Note that the function $f(n)$ can be chosen to be explicit;  an explicit choice for $f(n)$ is given in \cite[1.28]{TT} with the main term of
$(n!)^{5/2}$. Although this choice may seem to be inflated, it is noted in \cite[Remark 1.2(iii)]{BLST} that any choice of $f(n)$ should be
at least $\bmax(\SSS_n) > e^{-1.3\sqrt{n}}\sqrt{n!}$.

Let $\tilde\GC = \GL_n(K)$ and let $\a(\tilde\L)$ be the maximum value of $\frac{\dim u^{\tilde\L}}{\dim u^{\tilde\GC}}$ over nontrivial unipotent elements
$u\in \tilde\L$ (and $\a(\tilde\L) = 0$ if $\tilde\L$ is a torus). It is easy to see that $\a(\L) = \a(\tilde\L)$. Furthermore,
$\a(\tilde\L) \le \frac{n_m}{n}$ by \cite[Thm. 1.10]{BLST}. (Note that this bound is only stated for $\GL_n(q)$ in \cite[Theorem 1.10]{BLST}, but its proof applies to bound $\a(\tilde\L)$ for any proper Levi subgroup $\tilde\L$ of the algebraic group $\tilde\GC$.)
Hence $\a(L) \le \frac{n-s}{n}$, and the conclusion follows.
\end{proof}

The next lemma gives some properties of elements of $G$ of support $s$.

\begin{lem}\label{sest} For $1\le s<n$, define $N_s(G) = \{g\in \GSS : \nu(g)=s\}$ and let $n_s(G):=|N_s(G)|$.
\begin{itemize}
\item[{\rm (i)}] If $g \in N_s(g)$ and $s<\frac{n}{2}$ then $|\CB_G(g)|_p < q^{\frac{1}{2}n^2+s^2-ns}$.
\item[{\rm (ii)}] If $g \in N_s(g)$ and $s\ge \frac{n}{2}$ then $|\CB_G(g)|_p < q^{\frac{1}{2}(n^2-ns)}$.
\item[{\rm (iii)}] $\sum_{n-1 \geq s \geq n/2}n_s(G) < |G| < q^{n^2-1}$.
\item[{\rm (iv)}] If $s < n/2$, then $n_s(G) < cq^{s(2n-s)+n-1}$,
where $c$ is an absolute constant that can be taken to be 44.1.
\end{itemize}
\end{lem}

\begin{proof}
(i) Let  $g \in N_s(G)$ with $s<\frac{n}{2}$. Then $\hat g = {\rm diag}(\l I_{n-s},X)$ for some $\l \in \F_{q^u}^*$ (where $u=1$ if $\e=+$ and $u=2$ if $\e=-$)
and a suitable $s \times s$-matrix $X$, and
one can see that
\begin{equation}\label{cent1}
  \GL_{n-s}^\e(q) \leq \CB_{\GL_n(q)}(\hat g) \le  \GL_{n-s}^\e (q)\times \GL_s^\e (q).
\end{equation}
Now the statement follows, since $|\CB_G(g)|_p \leq |\CB_{\GL_n^\e (q)}(\hat g)|_p$.

\smallskip
(ii) Let  $g \in N_s(G)$ with $s\ge \frac{n}{2}$. Then
$$\CB_{\GL_n^\e (q)}(\hat g) = \prod^t_{i=1}\GL_{d_i}^{\e_i}(q^{k_i}),$$
where $n-s = d_1 \geq d_2 \geq \ldots \geq d_t \geq 1$ and $\sum^t_{i=1}d_ik_i = n$. Hence,
$|\CB_{\GL_n^\e (q)}(\hat g)|_p = q^D$, where
$$D := \sum^t_{i=1}k_id_i(d_i-1)/2 = \bigl(\sum^t_{i=1}k_id_i^2 -n\bigr)/2.$$
Using the obvious inequality $x^2+y^2 < (x+1)^2+(y-1)^2$ when $x \geq y$, we observe that,
over all $m$-tuples $(x_1 \geq x_2 \geq \ldots \geq x_m)$ of integers $0 \leq x_i \leq d_1$ and with fixed
$\sum^m_{i=1}x_i$, $\sum^m_{i=1}x_i^2$ is maximized when
$(x_1,x_2, \ldots,x_m)$ is $(d_1,d_1, \ldots ,d_1,e,0, \ldots,0)$ with $0 \leq e <d_1$.
Applying this observation to
$$(x_1,x_2, \ldots ,x_m) = \bigl(\underbrace{d_1, \ldots,d_1}_{k_1 {\mathrm {~times }}},
     \underbrace{d_2, \ldots,d_2}_{k_2 {\mathrm {~times }}},\ldots,\underbrace{d_t, \ldots,d_t}_{k_t {\mathrm {~times }}} \bigr)$$
(and $m = \sum^t_{i=1}k_i$), we see that
$$\sum^t_{i=1}k_id_i^2 \leq ad_1^2+b,$$
where $n = ad_1+b$ with $0 \leq b < d-1$. It follows that
$$2D \leq ad_1(d_1-1) < ad_1^2 \leq nd_1 = n(n-s),$$
and we are done as in (i).

\smallskip
(iii) This is obvious, since $|G| \leq |\SL_n^\e(q)| < q^{n^2-1}$.

\smallskip
(iv) By \cite[Lemma 4.1]{LMT},
$$\frac{9}{32}q^{n^2} < |\GL_n(q)| < |\GU_n(q)| \leq \frac{3}{2}q^{n^2}.$$
It now follows from \eqref{cent1} that
$$\begin{aligned}|g^G| & \leq |\hat g^{\GL^\e_n(q)}| = [\GL^\e_n(q):\CB_{\GL^\e_n(q)}(\hat g)] \\
  & \leq [\GL^\e_n(q):\GL^{\pm}_{n-s}(q)] < \frac{(3/2)q^{n^2}}{(9/32)q^{(n-s)^2}} = \tfrac{16}{3}q^{s(2n-s)}\end{aligned}$$
for any $g \in N_s(G)$.
Since the total number of conjugacy classes in $G$ is at most $8.26q^{n-1}$ by Propositions 3.6 and 3.10 of \cite{FG}, the statement follows.
\end{proof}

\begin{lem}\label{est}
Let $1\ne \c \in {\rm Irr}(G)$, and for $1\le s<n$, let $g_s \in N_s(G)$ be such that $|\chi(g_s)|$ is maximal. For  $l\ge 1$, define
\[
\D_l := \sum_{1 \leq s < n/2} cq^{ns+\frac{3n}{2}-1}\left|\frac{\chi(g_s)}{\chi(1)}\right|^l +
\sum_{n/2 \leq s < n}^{n-1} q^{n^2-\frac{1}{2}n(s-1)-1}\left|\frac{\chi(g_s)}{\chi(1)}\right|^l,
\]
with $c$ as in Lemma \ref{sest}.
If $\D_l<1$, then $[\chi^l,\St]_G\ne 0$.
\end{lem}

\begin{proof}
As in the proof of Lemma \ref{stdiam}, we have $[\chi^l,\St]_G\ne 0$ as long as $\Sigma_l < |G|_p$, where
\[
\Sigma_l := \sum_{1 \neq g\in \GSS}  \left|\frac{\chi(g)}{\chi(1)}\right|^l|\CB_G(g)|_p.
\]
Using Lemma \ref{sest}, we have
\[
\begin{aligned}
\Sigma_l & \leq  \sum^{n-1}_{s= 1}  n_s(G)\left|\frac{\chi(g_s)}{\chi(1)}\right|^l  |\CB_G(g)|_p \\
         & \le \sum_{1 \leq s < \frac{n}{2}} cq^{s(2n-s)+n-1} \left|\frac{\chi(g_s)}{\chi(1)}\right|^l q^{\frac{1}{2}n^2+s^2-ns} + \sum_{n/2 \leq s< n} q^{n^2-1} \left|\frac{\chi(g_s)}{\chi(1)}\right|^l q^{\frac{1}{2}(n^2-ns)}\\
         & = |G|_p\D_l,
\end{aligned}
\]
where $\D_l$ is as in the statement of the lemma.
The conclusion follows.
\end{proof}

\vspace{2mm}
\noindent {\bf Proof of Theorem \ref{diam}}

Adopt the notation of Lemma \ref{est}. By Theorem \ref{char}, for any $\c \in {\rm Irr}(G)$,
\[
|\chi(g_s)| < f(n)\chi(1)^{1-\frac{s}{n}}.
\]
Hence for $l\ge 1$,
\begin{equation}\label{for-d}
  \D_l < f(n)^l \left(\sum_{1 \leq s < n/2} cq^{ns+\frac{3n}{2}-1} \chi(1)^{-sl/n}+
\sum_{n/2 \leq s < n} q^{n^2-\frac{1}{2}n(s-1)-1}\chi(1)^{-sl/n}\right).
\end{equation}
Now we choose
$$l = 5\frac{\log |G|}{\log \chi(1)} = 5\frac{\log_q |G|}{\log_q \chi(1)}.$$
We claim that
\begin{equation}\label{for-l}
  8(n+2) \geq l > \frac{3n^2}{ \log_q \chi(1)}.
\end{equation}
This is obvious if $G = \PSL_2(q)$ is simple. If $G = \PSL_n(q)$ with $n \geq 3$, then by \cite[Theorem 1.1]{TZ1},
$$\chi(1) > q^{n-1},\mbox{ on the other hand, }q^{n^2-1} > |G| > q^{n^2-2}$$
(where the last inequality follows from \cite[Lemma 4.1(ii)]{LMT}),
and so \eqref{for-l} holds. If $G = \PSU_n(q)$ with $n \geq 3$, then again by \cite[Theorem 1.1]{TZ1},
$$\chi(1) > q^{n-2},\mbox{ on the other hand, }q^{n^2-1} > |G| > q^{n^2-3},$$
(where the last two inequalities can be checked using the proof of \cite[Lemma 4.1(iv)]{LMT}),
and so \eqref{for-l} holds.

Now \eqref{for-l} implies that  $\chi(1)^{-sl/n} < q^{-3ns}$. Hence the first summand inside the parenthesized part of \eqref{for-d} is
at most
$$c\sum_{1 \leq s < n/2}q^{3n/2-1-2ns} < cq^{-n/2-1}\sum^{\infty}_{j=0}\frac{1}{q^{2nj}} < \frac{16c}{15}q^{-n/2-1}.$$
The second summand inside the parenthesized part of \eqref{for-d} is
at most
$$\sum_{n/2 \leq s < n}q^{n^2-7ns/2+n/2-1} < \frac{n}{2}q^{-3n^2/4+n/2-1} \leq \frac{n}{2}q^{-n-1} < q^{-n/2-1}.$$
Since $c \leq 44.1$, it follows that
$$\D_l < f(n)^l\biggl(\frac{16c}{15}+1\biggr)q^{-n/2-1} < f(n)^l\biggl(\frac{49}{q}\biggr)^{n/2+1}.$$
Taking
$$q \geq (49f(n))^{16},$$
we obtain by \eqref{for-l} that $\D_l < 1$.
Hence $[\c^l,\St]_G\ne 0$ by Lemma \ref{est}.

Now Theorem \ref{diam} follows, using exactly the same argument as in the last paragraph of Section \ref{bddsec}.

\vspace{2mm}
\noindent {\bf Proof of Corollary \ref{diam-mf}}

Write $\a = \a_1 + \ldots + \a_k$, with $\a_i \in \Irr(G)$ and $\a_1(1) \leq \a_2(1) \ldots \leq \a_k(1)$. Since $\a$ is faithful,
$\a_k(1) \geq \dmin(G) >1$, where $\dmin(G)$ is the smallest degree of nontrivial irreducible characters of $G$; furthermore,
$k \leq k(G) := |\Irr(G)|$ as $\a$ is multiplicity-free. It is easy to check that $\dmin(G)^{1.5} > k(G)$ for $G = \PSL_2(q)$ with
$q \geq 11$. For $n \geq 3$ and $G = \PSL^\e_n(q)$, it follows from \cite[Theorem 1.1]{TZ1} and \cite[Propositions 3.6, 3.10]{FG} that
$$\dmin(G)^{3/2} \geq \biggl( \frac{q^n-q}{q+1} \biggr)^{3/2} \geq \biggl( \frac{5}{6}q^{n-1} \biggr)^{3/2} > 8.3q^{n-1} > k(G).$$
It follows that $\a_k(1)^{5/2} > k(G)\a_k(1) \geq k\a_k(1) \geq \a(1)$. By Theorem \ref{diam}, for some
$$N \leq C \frac{\log |G|}{\log \a_k(1)} < 2.5C \frac{\log |G|}{\log \a(1)},$$
we have $\sum^{N}_{i=0}\a_k^i$ contains all irreducible characters of $G$, whence $\sum^N_{i=0}\a^i$ also contains
all irreducible characters of $G$, i.e. $\diam \MC(G,a) \leq N$.

\section{Symmetric and alternating groups}\label{altsec}
This section is devoted to prove Theorem \ref{alt}.

As explained in the Introduction, $\diam \MC(G,\a)$ is at most $N = N(\a)$, if $N$ is the smallest positive integer such that
$\sum^N_{i=0}\a^i$ contains $\Irr(G)$.
Let $G := \SSS_n$, $S := \AAA_n$, and let $H:= \SSS_{n-1}$, $K := \SSS_{n-2} \times \SSS_2$,
$K' = \SSS_{n-2} < K$, and $L := \SSS_{n-3} \times \SSS_3$
be Young subgroups of $G$. If $\l \vdash n$ is a partition of $n$, let $\chi^\l$ denote the irreducible character of $\SSS_n$ labeled by
$\l$.

Given a faithful irreducible character $\a$ of $G$ or $H$,
we will now bound $N(\a)$ in a sequence of steps.

\smallskip
{\sc Step 1}. If $\a \in \Irr(G)$ and $\a = \chi^{(n-1,1)}$, then $N(\a) \leq n-1$.

Indeed, $\a$ takes $n$ distinct values $-1,0,1, \ldots ,n-3,n-1$, hence $N(\a) \leq n-1$ by \cite{Br}.

\smallskip
{\sc Step 2}. If $\a \in \Irr(G)$ and $\a|_H$ is reducible, then $N(\a) \leq 2n-2$. In particular,
$N(\chi^{(n-2,2)}) \leq 2n-2$. Likewise, if $n \geq 7$ and $\mu = (mu_1 \geq \mu_2 \geq \ldots \geq \mu_s \geq 1) \vdash n$
and $n-1 \geq \mu_1 \geq n-3$, then $N(\chi^\mu) \leq 2n-2$.

Indeed, by assumption we have that
$$2 \leq [\a|_H,\a|_H]_H = [\a^2|_H,1_H]_H = [\a^2,\Ind^G_H(1_H)]_G.$$
Recall that $\Ind^G_H(1_H) = 1_G + \chi^{(n-1,1)}$ and $[\a^2,1_G]_G = 1$. It follows that
$\a^2$ contains $\chi^{(n-1,1)}$, and so $N(\a) \leq 2n-2$ by Step 1.

The branching rule for complex representations of $\SSS_n$ implies that
$$\chi^{(n-2,2)}|_H = \chi^{(n-2,1)}+\chi^{(n-3,2)},$$
i.e. $\chi^{(n-2,2)}$ is reducible over $H$. Similarly, $\chi^\mu|_H$ is reducible for the $\mu$ listed above when $n \geq 7$,
whence we are done.

\smallskip
{\sc Step 3.} If $\a \in \Irr(G)$, then $N(\a) \leq 4n-4$.

Consider $K = \SSS_{n-2} \times \SSS_2$, where $\SSS_2 = \langle s \rangle$ is generated by a transposition $s$. If $\a|_K$ is
irreducible, then by Schur's Lemma $s$ acts as a scalar, and so $\a = 1_G$ or the sign character, contradicting the faithfulness of
$\a$. Thus $\a|_K$ is reducible, and so
$$2 \leq [\a|_K,\a|_K]_K = [\a^2|_K,1_K]_K = [\a^2,\Ind^G_K(1_K)]_G.$$
Recall that $\Ind^G_K(1_K) = 1_G + \chi^{(n-1,1)}+\chi^{(n-2,2)}$ and $[\a^2,1_G]_G = 1$.
If $\a^2$ contains $\chi^{(n-1,1)}$, then $N(\a) \leq 2n-2$ by Step 1. Otherwise
we must have that $\a^2$ contains $\chi^{(n-2,2)}$, and so
$N(\a) \leq 4n-4$ by Step 2.

\smallskip
From now on we will assume that $\a \in \Irr(S)$ and that $\a$ is an irreducible constituent of the restriction of $\chi=\chi^\l \in \Irr(G)$ to
$S$.

\smallskip
{\sc Step 4.} If $\a$ extends to $G$, then $N(\a) \leq 4n-4$.

Indeed, in this case $\a = \chi|_S$. By Step 3, $\sum^{4n-4}_{i=0}\chi^i$ contains $\chi^\mu$ for all $\mu \vdash n$. It follows
that $\sum^{4n-4}_{i=0}\a^i =  (\sum^{4n-4}_{i=0}\chi^i)|_S$ contains $\chi^\mu|_S$ for all $\mu \vdash n$, whence it
contains all $\Irr(S)$.

From now on, we will assume that $\a$ does not extend to $G$; equivalently, $\l$ is self-associated: $\l = \l^*$. For $n = 5,6$,
the character $\a$ takes at most $5$ different values on $S$, and so $N(\a) \leq 4$ by \cite{Br}. We will therefore assume $n \geq 7$.

\smallskip
{\sc Step 5.} If $\a$ is real-valued then $N(\a) \leq 4n-4$.

The assumption implies that $[\a^2|_S,1_S]=1$. Next, by inspecting the character table of $\AAA_5$, we see that any nontrivial complex
irreducible representation $\Phi$ of $\AAA_5$ affords all three distinct eigenvalues $1,\omega,\omega^2$ for the $3$-cycle $t= (1,2,3)$
($\omega \neq 1$ being a cubic root of unity in $\mathbb{C}$). We prove
by induction that the same statement holds for any $n \geq 5$. For the induction step $n \geq 6$, suppose $\Phi(t)$ affords at most two
distinct eigenvalues. By induction hypothesis, all composition factors of $\Phi|_{\AAA_{n-1}}$ are trivial. By Frobenius' reciprocity,
the character $\varphi$ of $\Phi$ is a constituent of
$$\Ind^S_{S \cap H}(1_{S \cap H}) = (\Ind^G_H(1_H))|_S = (\chi^n + \chi^{(n-1,1)})|_H,$$
and so $\varphi = \chi^{(n-1,1)}|_S$. But clearly in this case $\Phi(t)$ affords all three eigenvalues $1,\omega,\omega^2$, a contradiction.

Applying the established assertion to a complex representation $\Phi$ affording $\a$, we see that $\Phi(t)$ affords all three eigenvalues
$1,\omega,\omega^2$. We can choose the Young subgroup $L = \SSS_{n-3} \times \SSS_3$ such that $t \in \SSS_3 \cap L$, in which case
$\langle t \rangle \lhd L \cap S$. It follows that $\a|_{L \cap S}$ is reducible, and so
$$2 \leq [\a|_{S \cap L},\a|_{S \cap L}]_{S \cap L} = [\a^2|_{S \cap L},1_{S \cap L}]_{S \cap L} = [\a^2,\Ind^S_{S \cap L}(1_{S \cap L})]_S.$$
Observe that
$$\Ind^S_{S \cap L}(1_{S \cap L}) = (\Ind^G_L(1_L))|_S = 1_S + \sum^3_{i=1} \chi^{(n-i,i)}|_S,$$
and $\chi^{(n-i,i)}|_S$ is irreducible for $i \leq 3$. It follows that
$\a^2$ contains $\chi^{(n-j,j)}|_S$ for some $1 \leq j \leq 3$. As $N(\chi^{(n-j,j})) \leq 2n-2$ by Step 2, we have that $N(\a) \leq 4n-4$.

\smallskip
{\sc Step 6.} If $\a \neq \bar\a$ and $\l \neq (a^a)$ with $a \in \Z_{\geq 1}$, then $N(\a) \leq 2n-2$.

Since we are assuming that $\a$ does not extend to $G$ and $\chi^\l$ is real-valued, we have that $\chi^\l|_S = \a +\bar\a$ and that
$\l=\l^*$. Let $\mu$ be obtained from $\l$ by removing the last node of the shortest row of (the Young diagram of) $\l$. As
$\l \neq (a^a)$, observe that $\mu \neq \mu^*$. But $\l = \l^*$, so by symmetry we see that $\chi^\la|_H$ contains $\chi^\mu+\chi^{\mu^*}$.
The condition $\mu \neq \mu^*$ also implies that $\beta:=\chi^\mu|_{\AAA_{n-1}} = \chi^{\mu^*}|_{\AAA_{n-1}}$ is irreducible. It follows that
$\a|_{S \cap H}$ contains the real-valued irreducible character $\beta$, and so
$\a^2|_{S \cap H}$ contains $\beta^2$, which in turns contains $1_{S \cap H}$. Thus we have
$$1 \leq [\a^2|_{S \cap H},1_{S \cap H}]_{S \cap H} = [\a^2,\Ind^S_{S \cap H}(1_{S \cap H})]_S.$$
Now
$$\Ind^S_{S \cap H}(1_{S \cap H}) = (\Ind^G_H(1_H))|_S = 1_S + \chi^{(n-1,1)}|_S,$$
and $[\a^2,1_S]_S = 0$ since $\a \neq \bar\a$. Hence $\a^2$ contains $\chi^{(n-1,1)}|_S$, and
so $N(\a) \leq 2n-2$ by Step 1.

\smallskip
{\sc Final Step.} If $\a \neq \bar\a$ and $\l = (a^a)$ with $a \in \Z_{\geq 3}$, then $N(\a) \leq 4n-4$.

As in Step 6, since we are assuming that $\a$ does not extend to $G$ and $\chi^\l$ is real-valued, we have that $\chi^\l|_S = \a +\bar\a$. Let $\nu$ be obtained from $\l$ by removing the last two nodes of the last row of (the Young diagram of) $\l$, so that
$\nu \neq \nu^*$. But $\l = \l^*$, so by symmetry we see that $\chi^\lambda|_K$ contains $\chi^\nu+\chi^{\nu^*}$.
The condition $\nu \neq \nu^*$ also implies that $\gamma:=\chi^\nu|_{\AAA_{n-2}} = \chi^{\nu^*}|_{\AAA_{n-2}}$ is irreducible. It follows that
$\a|_{S \cap K'}$ contains the real-valued irreducible character $\gamma$, and so
$\a^2|_{S \cap K'}$ contains $\gamma^2$, which in turns contains $1_{S \cap K'}$. Thus we have
$$1 \leq [\a^2|_{S \cap K'},1_{S \cap K'}]_{S \cap K'} = [\a^2,\Ind^S_{S \cap K'}(1_{S \cap K'})]_S.$$
Now
$$\Ind^S_{S \cap K'}(1_{S \cap K'}) = (\Ind^G_{K'}(1_{K'}))|_S = 1_S + \chi^{(n-1,1)}|_S + \chi^{(n-2,2)}|_S + \chi^{(n-2,1^2)}|_S,$$
and $[\a^2,1_S]_S = 0$ since $\a \neq \bar\a$. Hence $\a^2$ contains at least one of (irreducible characters)
$\chi^{(n-1,1)}|_S$, $\chi^{(n-2,2)}|_S$, $\chi^{(n-2,1^2)}|_S$, and we
conclude that $N(\a) \leq 4n-4$ by Step 2.
\hal

\section{McKay graphs for quasi-simple groups}\label{qssec}

McKay graphs $\MC(G,\a)$ are usually considered for any finite group $G$ possessing a faithful character $\a$ (to guarantee connectedness).
In this section, we show that the diameters of McKay graphs for faithful irreducible characters of quasi-simple groups (with cyclic center)
can be bounded by the diameters of McKay graphs for simple groups.

\begin{thm}\label{qs}
Let $G$ be a finite quasi-simple group with cyclic center $\ZB(G)$, and let $\chi$ be a faithful irreducible character of $G$. Then there is a
nontrivial irreducible character $\beta$ of the simple group $S := G/\ZB(G)$ such that
$$\diam \MC(G,\chi) \leq |\ZB(G)| \cdot \diam \MC(S,\beta) + |\ZB(G)|-1.$$
In particular
$$\max_{\a \in \Irr(G),~\a~\mathrm{faithful}}\diam \MC(G,\a) \leq |\ZB(G)| \cdot \biggl(\max_{1_S \neq \gamma \in \Irr(S)}\diam \MC(S,\gamma)+1\biggr) -1.$$
\end{thm}

\pf
Let $e:=|\ZB(G)|$. Since $\mathrm{Ker}(\chi^e)$ contains $\ZB(G)$ but not $G$, we can find a nontrivial $\beta \in \Irr(S)$ such that $\beta$ inflated to
$G$ is an irreducible constituent of $\chi^e$. Now consider arbitrary $\varphi,\psi \in \Irr(G)$. Then there is $0 \leq i \leq e-1$ such that
the nontrivial character $\varphi\chi^i\overline\psi$ is trivial at $\ZB(G)$ and so contains a nontrivial $\delta \in \Irr(S)$. Thus
$$[\varphi\chi^i\overline\delta,\psi]_G = [\varphi\chi^i\overline\psi,\delta]_G > 0.$$
Next, we can find some $d \leq \diam \MC(S,\beta)$ such that $\beta^d$ contains $\overline\delta$. It follows that
$$[\varphi\chi^{i+de},\psi]_G \geq [\varphi\chi^i\beta^d,\psi]_G \geq [\varphi\chi^i\overline\delta,\psi]_G > 0,$$
i.e. a directed path of length $i+de$ connects $\varphi$ to $\psi$ in $\MC(G,\a)$.
\hal

As a final remark, we note that one cannot remove the term $|\ZB(G)|$ from the upper bound in Theorem \ref{qs}. Indeed, any directed path
connecting $1_G$ to any other $1_S \neq \psi \in \Irr(S)$ in $\MC(G,\a)$ must have length divisible by $|\ZB(G)|$.



\end{document}